\documentclass{article}

\usepackage[a4paper,left=1.2in,right=1.3in,top=1.2in,bottom=1.2in]{geometry}

\usepackage[british]{babel}
\usepackage{hyperref}

\newtheorem{theorem}{Theorem}
\def\qed{\unskip\nobreak\hfil\penalty50\hskip2em\hbox{}\nobreak
  \hfil\vrule width 1ex height 1ex\parfillskip0pt\finalhyphendemerits=0 \par}
\newenvironment{proof}{\smallskip\textbf{Proof.}}{\qed\smallbreak}

\usepackage{amsfonts,amsmath}
\usepackage[round]{natbib}
\usepackage{mhequ}

\newcommand{\e}{\mathrm{e}}
\def\email#1{\href{mailto:#1}{\texttt{#1}}}
\newcommand{\eps}{\varepsilon}
\newcommand{\eref}[1]{(\ref{#1})}
\newcommand*{\essinf}{\mathop{\mathrm{ess\,inf}}}
\newcommand{\ind}[1]{1_{#1}}
\newcommand{\lhs}{\hskip1cm&\hskip-1cm}
\newcommand{\lr}{{\underline r}}
\newcommand{\ls}{{\underline s}}
\newcommand{\N}{\mathbb{N}}
\def\quark{\setbox0\hbox{$x$}\hbox to\wd0{\hss$\cdot$\hss}}
\newcommand{\R}{\mathbb{R}}
\newcommand*{\Rpl}{\mathord{\R_{\scriptscriptstyle+}\!}}
\newcommand{\ur}{\bar{r}}
\newcommand{\us}{\bar{s}}
\newcommand{\W}{\mathbb{W}}

\newenvironment{jvlist}{\begin{list}{}{%
\setlength{\leftmargin}{2em}
\setlength{\labelwidth}{1em}
\setlength{\labelsep}{0.5em}
\setlength{\topsep}{0pt}
\setlength{\parsep}{\parskip}
\setlength{\itemsep}{0pt}}}{\end{list}}

\begin{document}

\title{Upper and Lower Bounds in Exponential Tauberian~Theorems}
\author{Jochen Voss\thanks{
University of Leeds,
Statistics Department,
Leeds LS2 9JT,
United Kingdom.
Email: \email{J.Voss@leeds.ac.uk}}}
\date{30th September 2009}
\maketitle

\begin{abstract}
  In this text we study, for positive random variables, the relation
  between the behaviour of the Laplace transform near infinity and the
  distribution near zero.  A result of De~Bruijn shows that
  $E(\e^{-\lambda X}) \sim \exp(r\lambda^\alpha)$ for
  $\lambda\to\infty$ and $P(X\leq\eps) \sim \exp(s/\eps^\beta)$ for
  $\eps\downarrow0$ are in some sense equivalent (for $1/\alpha
  = 1/\beta + 1$) and gives a relation
  between the constants~$r$ and~$s$.  We illustrate how this result
  can be used to obtain simple large deviation results.  For use in
  more complex situations we also give a generalisation of De~Bruijn's
  result to the case when the upper and lower limits are different
  from each other.
\end{abstract}

%%%%%%%%%%%%%%%%%%%%%%%%%%%%%%%%%%%%%%%%%%%%%%%%%%%%%%%%%%%%%%%%%%%%%%

\section{Introduction}

Tauberian theorems \citep[see for example the monograph][]{Ko04}
describe the connection between the behaviour of a positive random
variable near zero and the behaviour of its Laplace transform near
infinity.  From De~Bruijn's Tauberian theorem \citep[theorem~4.12.9
in][]{Bingham-Goldie-Teugels87} we can easily conclude the following
result.

\begin{theorem}\label{tauber}
  Let $X\geq0$ be a random variable on a probability space
  $(\Omega,\mathcal{A},P)$, $A\in\mathcal{A}$ an event with $P(A)>0$
  and $\alpha\in(0,1)$, $\beta>0$ with $\frac1\alpha = \frac1\beta+1$.
  Then the limit
  \begin{equ}[E:tauber-laplace]
    r = \lim_{\lambda\to\infty}
                \frac{1}{\lambda^\alpha} \log E(\e^{-\lambda X}\cdot 1_A) \leq 0
  \end{equ}
  exists if and only if
  \begin{equ}[E:tauber-prob]
    s = \lim_{\eps\to 0} \eps^\beta \log P(X\leq\eps,A) \leq 0
  \end{equ}
  exists and in this case we have $|\alpha r|^{1/\alpha} = |\beta
  s|^{1/\beta}$.
\end{theorem}

\begin{proof}
  In theorem~4.12.9 of \citet{Bingham-Goldie-Teugels87} choose their
  alpha to be $-1/\beta$, $\phi(x)=x^{-1/\beta}$,
  $\psi(x)=x^{-1/\alpha}$, and $B=|s|$.  This gives the proof in the
  case $A=\Omega$.  The case of general sets~$A$ can be reduced to the
  first case by considering the distribution $Q(\quark)=P(\quark\cap
  A)/P(A)$ instead of~$P$.
\end{proof}

With the help of this theorem we can use knowledge about the Laplace
transform of a given random variable~$X$ to show that the probability
$P(X\leq\eps)$ for $\eps\downarrow0$ decays exponentially fast.
Therefore in some situations Tauberian theorems of exponential type
can be valuable tools for deriving large deviation principles.
Typically, in this case one has $\alpha=1/2$, $\beta=1$ and thus
$s=-r^2/4$.  Section~2 illustrates this idea by using
theorem~\ref{tauber} to derive a simple large deviation result for the
conditional distribution of a Brownian motion, given that the
$L^2$-norm of the path is small.

In general, the limit~\eref{E:tauber-prob} does not necessarily exist.
For large deviation results one usually considers upper and lower
limits, and thus theorem~\ref{tauber} cannot be used directly.  In
section~3 of this text we will therefore derive a version of
theorem~\ref{tauber} which considers upper and lower limits.  A
(lengthy) application where upper and lower limits are needed, and
where theorem~\ref{tauber} is therefore not enough, can be found
in~\citet{Voss08}.

The special case of $\alpha=1/2$ and $\beta=1$ for the result
presented in this text was originally derived as part of my
PhD-thesis~\citep{Voss04}.

%%%%%%%%%%%%%%%%%%%%%%%%%%%%%%%%%%%%%%%%%%%%%%%%%%%%%%%%%%%%%%%%%%%%%%

\section[Brownian Paths with Small Norm]{Brownian Paths with Small $L^2$-Norm}

In this section we illustrate how theorem~\ref{tauber} can be used to
derive a simple large deviations principle (LDP) for Brownian motion.
See \cite{Dembo-Zeitouni98} for details about large deviations, and in
particular section 5.2 there for large deviation results for Brownian
motion.  A review of the connections between Tauberian theorems and
large deviations, and further references, can be found in~\cite{Bi08}.

Let $\mathcal{X}$ be the space of all paths
$\omega\colon [0,t]\to \R$ such that $\omega_0=0$, equipped with the
topology of pointwise convergence.  On $\mathcal{X}$, define a family
$(P_\eps)_{\eps>0}$ of measures by
\begin{equ}
  P_\eps(A) = \W\Bigl( A \Bigm| \int_0^t B_s^2 \,ds \leq\eps \Bigr)
\end{equ}
for all measurable $A\subseteq \mathcal{X}$, where $\W$ is the Wiener
measure on~$\mathcal{X}$ and~$B$ is the canonical process.

\begin{theorem}
On the space $\mathcal{X}$ the family $(P_\eps)_{\eps>0}$ satisfies
the LDP with the good rate function
\begin{equ}
I(\omega)
  = \sup \bigl\{ (t+2\omega_{t_1}^2+\cdots+2\omega_{t_n}^2 + \omega_t^2)^2/8-t^2/8
                        \bigm| n\in\N, 0<t_1<\cdots<t_n< t \}
\end{equ}
for all $\omega\in \mathcal{X}$.
\end{theorem}

\begin{proof}
Define $X = \int_0^t B_s^2 \,ds$.  In order to apply
theorem~\ref{tauber} we have to calculate the tails of the Laplace
transform of~$X$.  Formula~(1--1.9.7) from~\citet{Borodin-Salminen96}
states
\begin{equ}
E_x\Bigl(\exp\bigl(-\frac{\gamma^2}{2}\int_0^t B_s^2 \,ds\bigr);
        B_t\in dz \Bigr)
  = \varphi(x; t, z) \,dz
\end{equ}
where
\begin{equ}
\varphi(x; t, z)
  = \frac{\sqrt{\gamma}}{\sqrt{2\pi\sinh(t\gamma)}}
        \exp\Bigl( -\frac{(x^2+z^2)\gamma\cosh(t\gamma)-2xz\gamma}{2\sinh(t\gamma)} \Bigr).
\end{equ}
For starting point~$x$, measurable sets $A_1, \ldots, A_n \subseteq
\R$ and fixed times $0 < t_1 < \cdots < t_n = t$, the Markov property
of Brownian motion gives then
\begin{equs}[1]
\lhs E_x\Bigl(\exp\bigl(-\frac{\gamma^2}{2}\int_0^t B_s^2 \,ds\bigr)
        1_{A_1}(B_{t_1}) \cdots 1_{A_n}(B_{t_n}) \Bigr) \\
  &= \int_{A_1}\! \cdots \int_{A_n} \varphi(x;t_1,z_1) \varphi(z_1; t_2-t_1,z_2) \\
  &\hskip33mm
        \cdots \varphi(z_{n-1};t_n-t_{n-1},z_n) \; dz_n \cdots dz_1.
\end{equs}
We are interested in the exponential tails of this expression for
$\gamma\to\infty$.

Let $\eps>0$.  Observe that there are constants $0<c_1<c_2$ and $G>0$
with
\begin{equ}
c_1 \e^{-\gamma t/2}
  \leq \frac{1}{\sqrt{2\pi\sinh(\gamma t)}}
  \leq c_2 \e^{-\gamma t/2} \quad\mbox{for all $\gamma>G$.}
\end{equ}
Furthermore we can use the relation $|2xz|\leq x^2+z^2$ to get
\begin{equ}
\frac{x^2+z^2}{2}\cdot\frac{\cosh(\gamma t)-1}{\sinh(\gamma t)}
  \leq \frac{(x^2+z^2)\cosh(\gamma t)-2xz}{2\sinh(\gamma t)}
  \leq \frac{x^2+z^2}{2}\cdot\frac{\cosh(\gamma t)+1}{\sinh(\gamma t)}
\end{equ}
for all $x,z\in\R$.
Because of
\begin{equ}
\frac{\cosh(\gamma t)\pm 1}{\sinh(\gamma t)}
  = \frac{\e^{\gamma t}+\e^{-\gamma t}\pm 1}{\e^{\gamma t}-\e^{-\gamma t}}
  \longrightarrow 1 \quad\mbox{for $\gamma\to\infty$.}
\end{equ}
we can then find a $\gamma_0>0$, such that whenever $\gamma>\gamma_0$
the estimate
\begin{equ}
\frac{x^2+z^2}{2}\cdot(1-\eps)
  \leq \frac{(x^2+z^2)\cosh(\gamma t)-2xz}{2\sinh(\gamma t)}
  \leq \frac{x^2+z^2}{2}\cdot(1+\eps)
\end{equ}
holds for all $x,z\in\R$.

Using this estimate we conclude
\begin{equs}[1]
\lhs \limsup_{\gamma\to\infty} \frac{1}{\gamma} \log
  E_x\Bigl(\exp\bigl(-\frac{\gamma^2}{2}\int_0^t B_s^2 \,ds\bigr)
        1_{A_1}(B_{t_1}) \cdots 1_{A_n}(B_{t_n}) \Bigr) \\
  &\leq \lim_{\gamma\to\infty} \frac{1}{\gamma} \log \gamma^{n/2} c_2^n
 \int_{A_1}\! \cdots \int_{A_n}
    \e^{-\gamma t_1/2} \exp\Bigl(-\gamma\frac{x^2+z_1^2}{2}(1-\eps)\Bigr) \\
  & \hskip35mm \cdot
    \e^{-\gamma(t_2-t_1)/2} \exp\Bigl(-\gamma\frac{z_1^2+z_2^2}{2}(1-\eps)\Bigr) \cdot \, \cdots \\
  & \hskip35mm \cdot
    \e^{-\gamma(t_n-t_{n-1})/2} \exp\Bigl(-\gamma\frac{z_{n-1}^2+z_n^2}{2}(1-\eps)\Bigr)
 \; dz_n \cdots dz_1 \\
&= \lim_{\gamma\to\infty} \frac{1}{\gamma} \log
 \int_{A_1}\! \cdots \int_{A_n}
    \exp\bigl(-\gamma t_n/2 -\gamma(x^2/2+z_1^2+\cdots \\
  & \hskip50mm \cdots +
    z_{n-1}^2+z_n^2/2)(1-\eps)\bigr)
 \; dz_n \cdots dz_1.
\end{equs}
Note the special role of the final point~$z_n$.  With the help of the
Laplace principle \citep[see
\textit{e.g.}][section~4.3]{Dembo-Zeitouni98} we can calculate the
limit on the right hand side to get
\begin{align*}
\lhs \limsup_{\gamma\to\infty} \frac{1}{\gamma} \log
  E_x\Bigl(\exp\bigl(-\frac{\gamma^2}{2}\int_0^t B_s^2 \,ds\bigr)
        1_{A_1}(B_{t_1}) \cdots 1_{A_n}(B_{t_n}) \Bigr) \\
  &\leq - \essinf_{z_1\in A_1, \ldots, z_n\in A_n}
          \Bigl(t/2+(x^2/2+z_1^2+\cdots+z_{n-1}^2+z_n^2/2)(1-\eps)\Bigr).
\end{align*}
for all $\eps>0$ and thus
\begin{align*}
\lhs \limsup_{\gamma\to\infty} \frac{1}{\gamma} \log
  E_x\Bigl(\exp\bigl(-\frac{\gamma^2}{2}\int_0^t B_s^2 \,ds\bigr)
        1_{A_1}(B_{t_1}) \cdots 1_{A_n}(B_{t_n}) \Bigr) \\
  &\leq - \essinf_{z_1\in A_1, \ldots, z_n\in A_n}
          (t/2+x^2/2+z_1^2+\cdots+z_{n-1}^2+z_n^2/2).
\end{align*}

A very similar calculation gives
\begin{align*}
\lhs \liminf_{\gamma\to\infty} \frac{1}{\gamma} \log
  E_x\Bigl(\exp\bigl(-\frac{\gamma^2}{2}\int_0^t B_s^2 \,ds\bigr)
        1_{A_1}(B_{t_1}) \cdots 1_{A_n}(B_{t_n}) \Bigr) \\
  &\geq - \essinf_{z_1\in A_1, \ldots, z_n\in A_n}
          (t/2+x^2/2+z_1^2+\cdots+z_{n-1}^2+z_n^2/2).
\end{align*}
and together this shows
\begin{equs}[1][eq:bsqrintail]
\lhs \lim_{\gamma\to\infty} \frac{1}{\gamma} \log
  E_x\Bigl(\exp\bigl(-\frac{\gamma^2}{2}\int_0^t B_s^2 \,ds\bigr)
        1_{A_1}(B_{t_1}) \cdots 1_{A_n}(B_{t_n}) \Bigr) \\
  &= - \essinf_{z_1\in A_1, \ldots, z_n\in A_n}
          (t/2+x^2/2+z_1^2+\cdots+z_{n-1}^2+z_n^2/2).
\end{equs}

For measurable sets $A_1, \ldots, A_n \subseteq \R$ and fixed times $0
< t_1 < \cdots < t_n = t$, the Tauberian theorem~\ref{tauber} applied
to equation~\eref{eq:bsqrintail} gives
\begin{equs}[1][eq:condbbrate]
\lhs \lim_{\eps\downarrow0} \eps\cdot\log P\Bigl(
        (B_{t_1},B_{t_2},\ldots,B_{t_n})
        \in A_1\times A_2\times \cdots \times A_n
        \Bigm| \int_0^t B_s^2 \,ds \leq \eps \Bigr) \notag \\
  &= \lim_{\eps\downarrow0} \eps\cdot\log P\Bigl(
        B_{t_1}\in A_1,
        B_{t_2}\in A_2, \ldots,
        B_{t_n}\in A_n,
        \int_0^t B_s^2 \,ds \leq \eps \Bigr) \notag \\
  & \hskip 55mm - \lim_{\eps\downarrow0} \eps\cdot\log P\Bigl(
        \int_0^t B_s^2 \,ds \leq \eps \Bigr) \notag \\
  &= - \bigl(t+
                \essinf\limits_{z\in A_1\times A_2\times \cdots \times A_n}
          (2z_1^2+\cdots+2z_{n-1}^2+z_n^2)\bigl)^2/8  + t^2/8.
\end{equs}
Using $A_n=\R$ we can drop the assumption $t_n=t$ and arrive at the
following result.  For all measurable sets $A_1, \ldots, A_n \subseteq \R$
and fixed times $0 < t_1 < \cdots < t_n \leq t$ we have
\begin{equs}[E:pre-LDP]
\lhs \lim_{\eps\downarrow0} \eps\cdot\log P\Bigl(
        (B_{t_1},B_{t_2},\ldots,B_{t_n})
        \in A_1\times A_2\times \cdots \times A_n
        \Bigm| \int_0^t B_s^2 \,ds \leq \eps \Bigr) \\
  &= - \essinf\limits_{z\in A_1\times A_2\times \cdots \times A_n}
        I_{t_1,\ldots,t_n}(z)
\end{equs}
where $I_{t_1,\ldots,t_n}\colon \R^n\to \Rpl$ is defined by
\begin{equ}
I_{t_1,\ldots,t_n}(z)
  = \frac{1}{8}
\begin{cases}
\bigl(t+
          2z_1^2+\cdots+2z_n^2\bigl)^2 - t^2,& \text{if $t_n<t$, and} \\
        \bigl(t+
          2z_1^2+\cdots+2z_{n-1}^2+z_n^2\bigl)^2 - t^2& \text{for $t_n=t$.} \\
\end{cases}
\end{equ}

Since the rate function~$I_{t_1,\ldots,t_n}$ is continuous, we can
replace $\essinf$ with $\inf$ when the sets $A_i$ are open and
thus~\eref{E:pre-LDP} gives an LDP on~$\R^n$.  From this we can get
the LDP on the path space~$\mathcal{X}$ with rate function~$I$ by
applying the Dawson-G\"artner theorem about large deviations for
projective limits \citep[see for
example][theorem~4.6.1]{Dembo-Zeitouni98}.
\end{proof}

Note that the rate function~$I$ in the theorem will typically take its
infimum for a non-continuous path $\omega$: Assume $\omega$ is continuous
and non-zero.  Let $\eps=\|\omega\|_\infty / 2$.  Then we find infinitely
many distinct times $t$ with $\omega_{t}^2 > \eps^2$ and
thus~$I(\omega)=+\infty$.  Therefore it will not be possible to prove the
same theorem with $\mathcal{X}$ replaced by~$C\bigl([0,t],\R\bigr)$.

%%%%%%%%%%%%%%%%%%%%%%%%%%%%%%%%%%%%%%%%%%%%%%%%%%%%%%%%%%%%%%%%%%%%%%

\section{Upper and Lower Limits}

In this section we derive an analogue of theorem~\ref{tauber} which
considers upper and lower limits.  The proof does not rely on
theorem~\ref{tauber} and uses only elementary methods.

\begin{theorem}\label{tauberlimits}
  Let $X\geq0$ be a random variable on a probability space
  $(\Omega,\mathcal{A},P)$, $A\in\mathcal{A}$ an event with $P(A)>0$
  and $\alpha\in(0,1)$, $\beta>0$ with $\frac1\alpha = \frac1\beta+1$.
  \begin{jvlist}
  \item[a)] The upper limits
    \begin{equ}
      \ur = \limsup_{\lambda\to\infty}
                  \frac{1}{\lambda^\alpha} \log E(\e^{-\lambda X}\cdot1_A)
        \quad \mbox{and} \quad
      \us = \limsup_{\eps\to 0} \eps^\beta \log P(X\leq\eps,A)
    \end{equ}
    satisfy $|\alpha \ur|^{1/\alpha} = |\beta \us|^{1/\beta}$.
  \item[b)] The lower limits
    \begin{equ}
      \lr = \liminf_{\lambda\to\infty}
                  \frac{1}{\lambda^\alpha} \log E(\e^{-\lambda X}\cdot1_A)
        \quad \mbox{and} \quad
      \ls = \liminf_{\eps\to 0} \eps^\beta \log P(X\leq\eps,A).
    \end{equ}
    satisfy $|\alpha \lr|^{1/\alpha} \leq |\beta \ls|^{1/\beta} \leq
    |\e^{H(\alpha)} \alpha \lr|^{1/\alpha}$ where $H(\alpha) = -
    \alpha \log(\alpha) - (1-\alpha)\log(1-\alpha)$ and both bounds
    are sharp.
  \end{jvlist}
\end{theorem}

Note that, because $X$ is positive, the expectation $E(\e^{-\lambda
  X})$ exists for all $\lambda\geq 0$ and is a number between 0 and~1.
Thus the values $\ur$, $\lr$, $\us$, and $\ls$ will all be negative.
Also it is easy to see that theorem~\ref{tauberlimits} does not
directly imply theorem~\ref{tauber}: If the limit~$s$ from
theorem~\ref{tauber} exists, then we get
\begin{equ}
  |\beta s|^{1/\beta}
  = |\alpha\ur|^{1/\alpha}
  \leq |\alpha\lr|^{1/\alpha}
  \leq |\beta s|^{1/\beta}
\end{equ}
\textit{i.e.}\ the limit~$r$ also exists and satisfies $|\alpha
r|^{1/\alpha} = |\beta s|^{1/\beta}$.  But if we assume that~$r$
exists, then theorem~\ref{tauberlimits} only gives
\begin{equ}
   |\alpha r|^{1/\alpha}
  = |\beta\us|^{1/\beta}
  \leq |\beta\ls|^{1/\beta}
  \leq |\e^{H(\alpha)}\alpha r|^{1/\alpha}
\end{equ}
and we cannot directly conclude that the limit~$s$ from theorem~\ref{tauber}
exists.

\begin{proof}
As in the proof of theorem~\ref{tauber}, it is enough to consider
the case $A=\R$.  Throughout the proof we will use the relations
$\beta/\alpha = \beta+1$ and $\alpha/\beta = 1 - \alpha$ without
further comment.

a) The estimate $|\beta \us|^{1/\beta} \geq |\alpha \ur|^{1/\alpha}$
follows from the exponential Markov inequality: Let $\eps>0$.  From
\begin{equ}
  E(\e^{-\lambda X})
    \geq \e^{-\lambda\eps}
      P\bigl(\e^{-\lambda X}\geq\e^{-\lambda\eps}\bigr)
    = \e^{-\lambda\eps} P\bigl(X\leq\eps\bigr)
\end{equ}
we get $P(X\leq\eps) \leq \e^{\lambda\eps} E(\e^{-\lambda X})$ and
thus
\begin{equ}
  \eps^\beta \log P(X\leq\eps)
    \leq \eps^\beta\bigl(\lambda\eps + \log E(\e^{-\lambda X})\bigr)
        \quad \mbox{for all $\lambda\geq0$.}
\end{equ}
For $\lambda = (-\frac{\beta}{\beta+1}\ur)^{\beta+1}
\eps^{-(\beta+1)}$ the bound becomes
\begin{equ}
  \eps^\beta \log P(X\leq\eps)
  \leq \bigl(-\frac{\beta}{\beta+1}\ur\bigr)^{\beta+1} + \bigl(-\frac{\beta}{\beta+1}\ur\bigr)^\beta \frac{1}{\lambda^\alpha} \log E(\e^{-\lambda X}).
\end{equ}
Taking upper limits we get
\begin{align*}
  \us & = \limsup_{\eps\downarrow0} \eps\cdot\log P(X\leq\eps) \\
    &\leq \bigl(-\frac{\beta}{\beta+1}\ur\bigr)^{\beta+1} + \bigl(-\frac{\beta}{\beta+1}\ur\bigr)^\beta \ur
    = - \frac{\beta^\beta}{(\beta+1)^{\beta+1}} |\ur|^{\beta+1}
\end{align*}
and the claim follows by solving this inequality for $|\beta
\us|^{1/\beta}$.
\smallbreak

A more careful analysis is necessary to prove $|\beta \us|^{1/\beta}
\leq |\alpha \ur|^{1/\alpha}$.  We can express $\ur$ via the lower
tails of $X$:
\begin{align*}
\ur &= \limsup_{\lambda\to\infty}
              \frac{1}{\lambda^\alpha} \log E(\e^{-\lambda X}) \\
  &= \limsup_{\lambda\to\infty}
              \frac{1}{\lambda^\alpha} \log
              \int_0^1 P(\e^{-\lambda X} \geq t) \,dt \\
  &= \limsup_{\eps\downarrow0}
              \eps \log \int_0^\infty P(X \leq \eps^{1/\alpha}u) \e^{-u} \,du.
\end{align*}
The definition of $\us$ gives that for every $\delta > 0$ there exists
an $E>0$, such that for every $\eta<E$ we have $P(X\leq\eta) \leq
\exp\bigl((\us+\delta)/\eta^\beta\bigr)$.  Using this estimate and the
substitution $v = \eps u$ we find
\begin{equ}
  \ur \leq \limsup_{\eps\downarrow0} \eps \log
        \int_0^\infty \exp\bigl( - \frac{(|\us|-\delta)v^{-\beta}+v}{\eps} \bigr)
		\, \frac{1}{\eps} dv.
\end{equ}
The right-hand side can be evaluated by the Laplace principle again
and so we find
\begin{equ}
  \ur
  \leq -\essinf_{v\geq 0} \bigl( (|\us|-\delta)v^{-\beta}+v \bigr)
  = -\bigl(|\us|-\delta\bigr)^{1/(\beta+1)} \beta^{-\beta/(\beta+1)} (1+\beta)
\end{equ}
for every $0 < \delta < |\us|$ and thus
\begin{equ}
  |\ur|
  \geq |\us|^{1/(\beta+1)} \beta^{-\beta/(\beta+1)} (1+\beta)
  = |\us|^{\alpha/\beta} \frac{\beta^{\alpha/\beta}}{\alpha}.
\end{equ}
This completes the proof of the bound $|\alpha\ur|^{1/\alpha} \geq
|\beta\us|^{1/\beta}$.
\smallbreak

b) Replacing all upper limits with lower limits in the proof of
$|\beta \us|^{1/\beta} \geq |\alpha \ur|^{1/\alpha}$ gives the
corresponding bound $|\beta \ls|^{1/\beta} \geq |\alpha
\lr|^{1/\alpha}$.
\smallbreak

Finally, we prove $|\beta \ls|^{1/\beta} \leq | \e^{H(\alpha)}
\alpha \lr|^{1/\alpha}$, or equivalently $\lr \leq -
|\ls|^{1-\alpha}$:  Using the estimate $\e^{-\lambda x} \leq
\ind{[0,\eps]}(x) + \e^{-\lambda\eps} \ind{(\eps,\infty)}(x)$ for all
$x\geq 0$ gives $E\bigl(\e^{-\lambda X}\bigr) \leq P(X\leq\eps) +
\e^{-\lambda\eps}$.  Choosing $\eps = \eps(\lambda)$ such that
$1/\lambda^\alpha = |\ls|^{-\alpha}\eps^\beta$, we get
\begin{equ}
  \frac{1}{\lambda^\alpha} \log E\bigl(\e^{-\lambda X}\bigr)
  \leq |\ls|^{-\alpha}
        \eps^\beta \log\bigl(P(X\leq\eps) + \e^{-|\ls|\eps^{-\beta}}\bigr).
\end{equ}
For the second term in the sum, the limit $\lim_{\eps\downarrow0}
\eps^\beta \log \e^{-|\ls|\eps^{-\beta}} = -|\ls|$ exists and thus we
can conclude
\begin{align*}
\lr
  &\leq |\ls|^{-\alpha} \liminf_{\eps\downarrow0}
        \eps^\beta \log\Bigl(P(X\leq\eps) + \e^{-|\ls|\eps^{-\beta}}\Bigr) \\
  &= |\ls|^{-\alpha}
        \max\Bigl( \liminf_{\eps\downarrow0} \eps^\beta \log P(X\leq\eps) \,,\;
         \lim_{\eps\downarrow0} \eps^\beta \log \e^{-|\ls|\eps^{-\beta}} \Bigr) \\
  &= |\ls|^{-\alpha} \max\bigl(-|\ls| \,,\; -|\ls|\bigr)
   = -|\ls|^{1-\alpha}.
\end{align*}
This is the required result.

The lower bound on $|\ls|$ is sharp, because in the case of
theorem~\ref{tauber} we have equality there.  The fact that the upper
bound on~$|\ls|$ is sharp is shown by the following example.
\end{proof}

\textbf{Example.}  This example illustrates that the bound $|\beta
\ls|^{1/\beta} \leq |\e^{H(\alpha)} \alpha \lr|^{1/\alpha}$ is sharp.
Let $s<0$, $\alpha$ and $\beta$ as above, and $(\eps_n)_{n\in\N_0}$ be
a strictly decreasing sequence with $\eps_0=\infty$ and
$\lim_{n\to\infty} \eps_n = 0$.  Then we have
\begin{equ}
  \sum_{n=1}^\infty \Bigl(\e^{-|s|\eps_{n-1}^{-\beta}}-\e^{-|s|\eps_n^{-\beta}}\Bigr)
    = \e^{-|s|\eps_0^{-\beta}} - \lim_{n\to\infty}\e^{-|s|\eps_n^{-\beta}}
    = 1 - 0
    = 1
\end{equ}
and we can define a random variable~$X$ with values in the set $\{\,
\eps_n \mid n\in\N \,\}$ by
\begin{equ}
  P(X=\eps_n) = \e^{-|s|\eps_{n-1}^{-\beta}}-\e^{-|s|\eps_n^{-\beta}}
\end{equ}
for all $n\in\N$.  This random variable has
\begin{equ}
  P(X\leq \eps)
    = \sum_{n=n(\eps)}^\infty
        \Bigl(\e^{-|s|\eps_{n-1}^{-\beta}}-\e^{-|s|\eps_n^{-\beta}}\Bigr)
    = \e^{-|s|\eps_{n(\eps)-1}^{-\beta}}
\end{equ}
with $n(\eps)=\min \{\, n\in\N \mid \eps_n\leq\eps \,\}$ and consequently
\begin{equ}
  \eps^\beta \log P(X\leq \eps) = -|s|\frac{\eps^\beta}{\eps_{n(\eps)-1}^\beta}.
\end{equ}
By definition of $n(\eps)$ we have $\eps_{n(\eps)} \leq \eps <
\eps_{n(\eps)-1}$.  This allows us to calculate the exponential tail
rates $\ls=s$ and, because $s$ is negative,
$\us=s\cdot\liminf_{n\to\infty} \eps_n^\beta/\eps_{n-1}^\beta$.

Choosing different sequences $(\eps_n)$ leads to different values for
$\us$, $\ur$, and $\lr$.  For our example let $q<1$ and define
$\eps_n=q^n$ for all $n\in\N$.  Then the above calculation shows
$\us=qs$ and $\ls=s$.  Theorem~\ref{tauberlimits} gives $\ur = -|\beta
qs|^{\alpha/\beta}/\alpha$ and $\lr \in \bigl[ -|\beta
s|^{\alpha/\beta}/\alpha, -\e^{-H(\alpha)} |\beta
s|^{\alpha/\beta}/\alpha\bigr] = \bigl[ -
\e^{H(\alpha)}|s|^{\alpha/\beta}, - |s|^{\alpha/\beta}\bigr]$.  In
order to show that the upper bound on $|\ls|$ is sharp, we have to
show that we can have $\lr$ arbitrarily close
to~$-|s|^{\alpha/\beta}$.

In the situation of the example we can get better bounds on~$\lr$ by
an explicit calculation.  The Laplace transform of~$X$ is given by
\begin{align*}
E(\e^{-\lambda X})
  &= \sum_{n\in\N} \e^{-\lambda q^n}
        \bigl( \e^{-|s|q^{-\beta(n-1)}}-\e^{-|s|q^{-\beta n}} \bigr) \\
  &= \sum_{n\in\N} \e^{-\lambda q^n-|s|q^{-\beta(n-1)}}
        \bigl( 1-\e^{-|s|(1-q^\beta)q^{-\beta n}} \bigr).
\end{align*}
Since $\exp\bigl(-|s|(1-q^\beta)q^{-\beta n}\bigr)\to 0$ as
$n\to\infty$, we have $1/2 < 1 - \exp\bigl(-|s|(1-q^\beta)q^{-\beta
  n}\bigr) < 1$ for sufficiently large~$n$.  Define $n(\lambda)$ by
$q^{n(\lambda)} \in \bigl[ q (|s|/\lambda)^{\alpha/\beta},
(|s|/\lambda)^{\alpha/\beta} \bigr)$.  With $f(x)=\exp(-\lambda x -
q^\beta|s|x^{-\beta})$ we have
\begin{equ}
E(\e^{-\lambda X})
  > \exp\bigl( -\lambda q^{n(\lambda)}
                        - |s|q^{-\beta(n(\lambda)-1)} \bigr) \frac{1}{2}
  = \frac{1}{2} f(q^{n(\lambda)})
\end{equ}
for sufficiently large $\lambda$.  Because the only local extremum of
$f$ is a local maximum, we can get a lower bound for $f$ on the
interval $\bigl[ q (|s|/\lambda)^{\alpha/\beta},
(|s|/\lambda)^{\alpha/\beta} \bigr)$ by just considering the boundary
points.  This leads to
\begin{equ}
E(\e^{-\lambda X})
  > \frac{1}{2} \min\bigl(f(q (|s|/\lambda)^{\alpha/\beta}),
                                f((|s|/\lambda)^{\alpha/\beta}) \bigr)
  = \frac{1}{2} \exp\bigl( -(1+\max(q,q^\beta))\lambda^\alpha|s|^{\alpha/\beta} \bigr)
\end{equ}
for sufficiently large $\lambda$.  Taking lower limits we get
\begin{equ}
  \lr
  = \liminf_{\lambda\to\infty}
                  \frac{1}{\lambda^\alpha} \log E(\e^{-\lambda X})
  \geq -\bigl(1 + \max(q,q^\beta)\bigr)|s|^{\alpha/\beta}.
\end{equ}
By choosing small values of~$q$, we can force $\lr$ to be arbitrarily
close to $-|s|^{\alpha/\beta}$ and thus the bound from the theorem is
sharp.

%%%%%%%%%%%%%%%%%%%%%%%%%%%%%%%%%%%%%%%%%%%%%%%%%%%%%%%%%%%%%%%%%%%%%%

\bigbreak

\textbf{Acknowledgements.} I want to thank the anonymous referees for
pointing out that my original proof for the case $\alpha=1/2$ and
$\beta=1$ could be changed to give the more general result presented
here, and also for pointing me to the references \citet{Ko04}
and~\citet{Bi08}.

\bibliographystyle{plainnat-jv}
\bibliography{tauber}

\end{document}